\documentclass[12pt]{elsarticle}

\usepackage[nottoc,notlot,notlof]{tocbibind}

\usepackage{amsfonts,amssymb,amsmath}
\usepackage{lineno,hyperref}

\usepackage{indentfirst}

\makeatletter
\def\ps@pprintTitle{%
   \let\@oddhead\@empty
   \let\@evenhead\@empty
   \let\@oddfoot\@empty
   \let\@evenfoot\@oddfoot
}
\makeatother

\newtheorem{thm}{Theorem}[section]

\newtheorem{lem}[thm]{Lemma}
\newtheorem{prop}[thm]{Proposition}
\newtheorem{defn}[thm]{Definition}
\newtheorem{rem}[thm]{Remark}
\newtheorem{ex}{Example}

\newenvironment{proof}[1][\noindent \textbf{Proof: }]{#1}{ \hfill $\square$ \vspace{2mm}}

\begin{document}

\begin{frontmatter}

\title{Global hypoellipticity for a class of\\ periodic Cauchy operators}

\author{Fernando de \'{A}vila Silva}
\ead{fernando.avila@ufpr.br}

\address{Departamento de Matem\'{a}tica, Universidade Federal do Paran\'{a}, \\ Caixa Postal 19081, Curitiba, PR 81531-990, Brazil}

\begin{abstract}

This note presents an investigation on the  global hypoellipticity problem for Cauchy operators on $\mathbb{T}^{n+1}$ belonging to the class  \linebreak
$L = \prod_{j=1}^{m}\left(D_t + c_j(t) P_j(D_x)\right)$, where $P_j(D_x)$ is a pseudo-differential operator on $\mathbb{T}^n$ and $c_j = c_j(t)$, a smooth, complex valued  function on $\mathbb{T}$. The main goal of this investigation consists in establishing connections between the global hypoellipticity of the operators $L$ and its normal form $L_0 =  \prod_{j=1}^m \left( D_t + c_{0,j}P_j(D_x)\right)$. In order to do so, the problem is approached   by combining H\"{o}rmander's and Siegel's  conditions on the symbols of the operators $L_j = D_t + c_j(t) P_j(D_x)$.

\end{abstract}


\end{frontmatter}

%%%%%%%%%%%%%%%%%%%%%%%%%%%%%%%%%%%%%%%%%%%%%%%%%%%%%%%%%%%%%%%%
%%%%%%%%%%%%%%%%%%%%%%%%%%%%%%%%%%%%%%%%%%%%%%%%%%%%%%%%%%%%%%%%
\section{Introduction}
%%%%%%%%%%%%%%%%%%%%%%%%%%%%%%%%%%%%%%%%%%%%%%%%%%%%%%%%%%%%%%%%
%%%%%%%%%%%%%%%%%%%%%%%%%%%%%%%%%%%%%%%%%%%%%%%%%%%%%%%%%%%%%%%%

This note discusses  the  global hypoellipticity problem for operators 
belonging to the class
\begin{equation}\label{Main-operator}
L  = \prod_{j=1}^{m}\left(D_t + c_j(t) P_j(D_x)\right), \ (t, x) \in \mathbb{T} \times \mathbb{T}^n,
\end{equation}
where $\mathbb{T} \simeq \mathbb{R}/ 2 \pi\mathbb{Z}$ is the  Torus and  $P_j(D_x)$ is a pseudo-differential operator of order $\nu_j \in \mathbb{R}$ defined by 
$$
P_j(D_x) u(x) =  \sum_{\xi \in \mathbb{Z}^n}{e^{i x \cdot \xi} p_j(\xi) \widehat{u}(\xi)},
$$
where its  symbol $p_j=p_j(\xi)$  satisfies
\begin{equation}\label{bound-symb}
|p_j(\xi)| \leq C |\xi|^{\nu_j}, \ \forall \xi \in \mathbb{Z}^{n}
\end{equation}
and $\widehat{u}(\xi)$ are the Fourier coefficients
$$
\widehat{u}(\xi) = (2\pi)^{-n} \int_{\mathbb{T}^n}{e^{- i x \cdot \xi} u(x) dx}, \ \xi \in \mathbb{Z}^n.
$$

In particular, $p_j$ belongs to the class $S^{\nu_j}(\mathbb{Z}^n)$,  as shown in Ruzhansky and Turunen \cite{RT3}. 

Let us recall that a pseudo-differential operator $\mathcal{P}$ is globally hypoellitpic ((GH) for short) on the $N$-dimensional torus  $\mathbb{T}^{N}$   if conditions $u \in \mathcal{D}'(\mathbb{T}^{N})$  and $\mathcal{P} u  \in C^{\infty}(\mathbb{T}^{N})$ imply $u \in C^{\infty}(\mathbb{T}^{N})$.

We emphasize that the study of global hypoellipticity is a highly non-trivial problem
and it seems impossible to attack it by a unified approach, even for the classes of vector fields, as we can see in A. Bergamasco, P. Cordaro and G. Petronilho \cite{BCP04}, G. Petronilho \cite{Petr11}, A. Himonas and G. Petronilho \cite{HIMONASGER} and J. Hounie \cite{HOU79}. Also, the study is considered in different functional spaces, for instance,  analytic and Gevrey classes, see A. Bergamasco \cite{BERG99}, A. Albanese \cite{ALBANESE2004},  G. Chinni, P. Cordaro  \cite{C-CHINI}  and the references therein.

Furthermore, the problem is analyzed in the class of pseudo-differential operators as proposed by D. Dickson, T. Gramchev and M. Yoshino  in \cite{DGY02} where 
the authors characterizes the global hypoellipticity by means of Siegel type conditions on the symbols of these operators.

Now, since the main goal of this investigation is to establish connections between the global hypoellipticity of  $L$ and its normal form
$$
L_0 =  \prod_{j=1}^m \left( D_t + c_{0,j}P_j(D_x)\right), \ (t, x) \in \mathbb{T}^{n+1}, 
$$
where $c_{0,j} = (2 \pi)^{-1} \int_{0}^{2 \pi}c_{j}(t)dt$, first we take inspiration
from \cite{AGKM}, where R. Gonzalez, A.  Kirilov, C. Medeira and F. {\'A}vila, present a characterization of the  global hypoellipticity for operators in the form
$$
L_j = D_t+c_j(t)P_j(D_x),  (t, x) \in \mathbb{T}^{n+1},
$$
by analyzing the functions 
$$
t\in\mathbb{T}\mapsto \mathcal{M}_j(t,\xi)\doteq c_j(t)p_j(\xi), \ \xi \in \mathbb{Z}^n,
$$
and its averages
$$
{\mathcal{M}}_{0,j}(\xi) = (2 \pi)^{-1} \int_{0}^{2 \pi} {\mathcal{M}_j}(t, \xi) dt, \ \xi \in \mathbb{Z}^n.
$$

Secondly, by denoting 
$Z_{{\mathcal{M}}_j} = \{\xi \in \mathbb{Z}^n; \, {\mathcal{M}}_{0,j}(\xi) \in\mathbb{Z} \}$ and $L_{0,j} = D_t + c_{0,j}P_j(D_x),$ then the  results from \cite{AGKM}, that  relate to this note,  can be summarized by the following:

\begin{thm}\label{AGKM} Set $c_j(t) = a_j(t) + i b_j(t)$ and 
	$p_j(\xi) = \alpha_j(\xi) + i \beta_j(\xi)$.
	
	\begin{enumerate}
		\item [(a)] if $L_j$ is globally hypoelliptic, then  $Z_{\mathcal{M}_j}$ is finite and $L_{0,j}$ is globally hypoelliptic;

		\item [(b)] if $L_{0,j}$ is  globally hypoelliptic  and the functions 
		$$
		\mathbb{T} \ni t \mapsto \Im\mathcal{M}_j(t,\xi) = a_j(t)\beta_j(\xi) + b_j(t)\alpha_j(\xi) 
		$$
		do not change sign, for a sufficiently large $|\xi|$, then $L_j$ is (GH).
		
		\item [(c)]  the following statements are equivalent:
		\begin{itemize}
			\item[i)] $L_{0,j}$  is globally hypoelliptic;
			\item[ii)]  there exist positive constants $C_j$, $M_j$ and $R_j$ such that
			\[|\tau+\mathcal{M}_{0,j}(\xi)|\geqslant C_j|(\tau,\xi)|^{-M_j}, \ \textrm{for all} \ |\tau| + |\xi| \geqslant R_j;\]
			\item[iii)] there exist positive constants $\widetilde{C}_j$, $\widetilde{M}_j$ and $\widetilde{R}_j$ such that
			\[|1-e^{\pm2\pi i\mathcal{M}_{0,j}(\xi)}|\geqslant \widetilde{C}_j|\xi|^{-\widetilde{M}_j}, \ \textrm{for all} \  |\xi|\geqslant \widetilde{R}_j.\]
		\end{itemize}
		
	\end{enumerate} 
	
\end{thm}

In view of the main goal, the first step in our investigations is the study of 
the constant coefficients operators belonging to the class:
$$
L =  \prod_{j=1}^m \left( D_t + \alpha_j P_j(D_x)\right), \ (t, x) \in \mathbb{T}^{n+1}, \ \alpha_j \in \mathbb{C}.
$$

In Section \ref{section2}, we present an analysis for this class that combines Theorem \ref{AGKM} and a generalization of Siegel's condition. This condition was introduced by D. Dickson, T. Gramchev and M. Yoshino in \cite{DGY02}. As it shall be stated in Theorem  \ref{main-the-constant}, necessary and sufficient conditions for global hypoellipticity are presented  by analyzing the behavior of the symbol $L(\tau, \xi)$.

In section \ref{section3}, the case of variable coefficients is studied.
A first investigation is performed through  Theorem \ref{AGKM} and also by assuming the following commutative condition:
$$
[L_j, L_k] = L_j  L_k - L_k  L_j = 0, \ j, k \in \{ 1, \ldots m\},
$$
that will be shown in  Theorem \ref{theorem-vari-1}.

In section \ref{section3.1}, a second investigation is presented and the approach is motivated by a condition proposed in H\"{o}rmander \cite{Hormander3}: there exists positive constants $\eta_j$ such that
\begin{equation}\label{H-hypo}
\Im \mathcal{M}_j(t, \xi) \geq -\eta_j, \ t \in \mathbb{T}, \  \forall \xi \in \mathbb{Z}^n.
\end{equation}

Under this hypothesis, in  Theorem \ref{Th-GH-H}  necessary and sufficient conditions for the global hypoellipticity of the operator $L_j$ will be exhibited.
Moreover, in Theorem \ref{Theorem8}, connections between the hypoellipticity of operators $L$ and $L_0$ will be established. Also, in Theorem \ref{Th-Solv}
conditions for the global  solvability of the operator $L_j$ will be presented and, in particular, obstructions for the global hypoellipticity of operator $L$ will be shown in Theorem \ref{suf-non-hy}.

Furthermore, also in section \ref{section3.1}, an interesting fact shall present itself: in view of  condition \eqref{H-hypo}, it becomes possible to obtain a new class of globally hypoelliptic  operators, so that the functions $\Im\mathcal{M}_j(\cdot,\xi)$ change sign for infinitely many indexes $\xi \in \mathbb{Z}^n$. This fact will be discussed at the beginning of the section, where an example will also be proposed.

%%%%%%%%%%%%%%%%%%%%%%%%%%%%%%%%%%%%%%%%%%%%%%%%%%%%%%%%%%%%%%%%
%%%%%%%%%%%%%%%%%%%%%%%%%%%%%%%%%%%%%%%%%%%%%%%%%%%%%%%%%%%%%%%%
\section{Operators with constant coefficients \label{section2}}
%%%%%%%%%%%%%%%%%%%%%%%%%%%%%%%%%%%%%%%%%%%%%%%%%%%%%%%%%%%%%%%%
%%%%%%%%%%%%%%%%%%%%%%%%%%%%%%%%%%%%%%%%%%%%%%%%%%%%%%%%%%%%%%%%

In this section we consider the constant coefficient  operator 
\begin{equation}\label{cc}
L =  \prod_{j=1}^m \left( D_t + \alpha_j P_j(D_x)\right), \ (t, x) \in \mathbb{T}^{n+1}, \ \alpha_j \in \mathbb{C},
\end{equation}
and study the global hypoellipticity by analyzing the growth of its symbol
\begin{equation*}
L(\tau, \xi) =   \prod_{j=1}^{m}(\tau - \alpha_j p_j(\xi)), \ (\tau, \xi) \in \mathbb{Z} \times \mathbb{Z}^n.
\end{equation*}

By following the approach introduced by Greenfield and Wallach in \cite{GW1}, we shall characterize the global hypoellipticity of $L$
by means of a control in $L(\tau, \xi)$:

\begin{thm}\label{GW}
	The operator $L$ is globally hypoelliptic if and only if there are positive constants $C$, $M$ and $R$ so that
	\begin{equation*}
	|L(\tau, \xi)|  \geq C |(\tau,\xi)|^{-M}, \ |\tau| + |\xi| \geq R.
	\end{equation*}	
\end{thm}

Now, in order to present a complete characterization for the global hypoellipticity, we introduce a Siegel condition,  inspired by  D. Dickson, T. Gramchev and M. Yoshino in \cite{DGY02}.

\begin{defn}
	We say that a sequence $\{c(\xi)\}_{\xi\in \mathbb{Z}^n}$ satisfies the \textit{generalized Siegel condition} if there exists positive constants $C$, $M$ and $R$  such that
	\begin{equation}\label{Siegel}
	|\tau - c(\xi)| \geq C |(\tau, \xi)|^{-M}, \ |\tau|+| \xi| \geq R, \ (\tau, \xi) \in \mathbb{Z}\times \mathbb{Z}^n.	
	\end{equation}
	
	Furthermore, a set $\Gamma=\{ \{c_j(\xi)\}_{\xi\in \mathbb{Z}^n}, \ j = 1, \ldots, m \}$
	satisfies  the \textit{simultaneous generalized Siegel condition}  if each $\{c_j(\xi)\}_{\xi\in \mathbb{Z}^n}$ satisfies \eqref{Siegel}. In this case we can write  $\Gamma \in  (\mathcal{GS})_{\infty}$.
\end{defn}

\begin{thm}\label{main-the-constant}
	Let $L$ be the operator defined in \eqref{cc} and set 
	$$
	\sigma_{L} = \{\{\alpha_j p_j(\xi)\}_{\xi \in \mathbb{Z}^n}, \ j = 1, \ldots, m \}.
	$$
	
	Then, $L$ is globally hypoelliptic if and only if $\sigma_{L} \in (\mathcal{GS})_{\infty}$.  
\end{thm}

\begin{proof}
	Since the coefficients of $L_j$ are constant,  we obtain $[L_j, L_k] = 0$, for 
	$j,k \in\{1, \ldots, m\}$. Hence, $L$ is globally hypoelliptic if and only if each $L_j$ is also globally hypoelliptic.  Then, if $L$ is (GH),  it follows from Theorem \ref{AGKM} that all sequences $\{\alpha_j p_j(\xi)\}$ satisfy \eqref{Siegel}.

	Conversely, if $\sigma_L \in  (\mathcal{GS})_{\infty}$ it is possible to  obtain constants $C_j$, $R_j$ and $M_j$ so that $|\tau - \alpha_j p_j(\xi)| \geq C_j |(\tau,\xi)|^{-M_j}$, for all $|\tau| + |\xi|\geq R_j$, hence
	$$
	|L(\tau, \xi)| = \prod_{j=1}^{m}|\tau - \alpha_j p_j(\xi)| \geq   C |(\tau,\xi)|^{-M}, \ |\tau| + |\xi|\geq R,
	$$
	where $C = \prod C_j$, $M = \sum M_j$  and $R = \max R_j$. The proof is then complete, in view of Theorem \ref{GW}.	
	
\end{proof}

\begin{ex}
	Let $L$ be 	the differential operator 
	\begin{equation*}
	L =  \prod_{j=1}^{m} (D_t - \alpha_j D_x), \ (t,x) \in \mathbb{T}^2.
	\end{equation*}
	
	In this case, it is possible to write:
	\begin{equation*}
	L(\tau, \xi) =  (\tau - \alpha_1 \xi) \ldots (\tau - \alpha_{m} \xi)
	\end{equation*}
	and thus $L$ is not globally hypoelliptic if and only if one of the numbers $\alpha_j$ is either a rational  or  Liouville. It can be noted here that this example recaptures results in \cite{BERGZANI}.
\end{ex}

\begin{rem}
	It is possible to consider operators belonging to the class
	\begin{equation*}
	\mathcal{P} =D_t^m + \sum_{j=1}^{m} {P_j(D_x) D_t^{m-j} }, \ (t, x) \in \mathbb{T}^{n+1}, \ c_j \in \mathbb{C},
	\end{equation*}
	rewriting  their symbols as
	\begin{equation*}
	\mathcal{P} (\tau, \xi) =   \prod_{j=1}^{m}(\tau - \sigma_{j}(\xi)), \ (\tau, \xi) \in \mathbb{Z} \times \mathbb{Z}^n,
	\end{equation*}
	where $\sigma_j(\xi)$ satisfies $\mathcal{P}(\sigma_{j}(\xi), \xi) = 0$. It follows from Theorem \ref{main-the-constant} that $\mathcal{P} $ is globally hypoelliptic provided that
	$$
	\sigma_{\mathcal{P}} = \{\{\sigma_j(\xi)\}_{\xi \in \mathbb{Z}^n}, \ j = 1, \ldots, m \} \in (\mathcal{GS})_{\infty}.
	$$
	
\end{rem}

\begin{ex}
	Consider $\Delta_x = -\sum_{j=1}^{n}\partial^2_{x_j}$ defined on $\mathbb{T}^n$ and set 
	$$
	\mathcal{P} = D^{2}_t - 2\alpha D_t (-\Delta_x)^{1/2} + \beta^2 \Delta_x,
	$$
	with $\alpha, \beta \in \mathbb{R}$. In this case, 
	$\mathcal{P}(\tau, \xi)   =  \left( \tau  - |\xi| \rho_1 \right) \left(  \tau  - |\xi| \rho_2 \right)$, 	where
	$$
	\rho_{1,2} =  -\alpha \pm \sqrt{\alpha^2 - \beta^2}.
	$$

	If 	$\alpha^2 - \beta^2 <0$, then $\rho_{1,2}$  is not real  and $\mathcal{P}$ is globally hypoelliptic. However, the case of real roots could be much more complicated since, in general, the usual approximations by rational numbers
	can't be applied, once
	$\tau / |\xi|$ may be irrational. For instance, when $\alpha=\beta$ we have
	$\mathcal{P} = (D_t - \alpha (-\Delta_x)^{1/2})^2$ and 
	$$
	\mathcal{P}(\tau, \xi) = |\xi|^2\left(\dfrac{\tau}{|\xi|}  - \alpha  \right)^2.
	$$
	
	The reader can find a complete discussion and examples for this type of approximations in \cite{AGKM}.
	
\end{ex}

%%%%%%%%%%%%%%%%%%%%%%%%%%%%%%%%%%%%%%%%%%%%%%%%%%%%%%%%%%%%%%%%
%%%%%%%%%%%%%%%%%%%%%%%%%%%%%%%%%%%%%%%%%%%%%%%%%%%%%%%%%%%%%%%%
\section{Operators with variable coefficients \label{section3}}
%%%%%%%%%%%%%%%%%%%%%%%%%%%%%%%%%%%%%%%%%%%%%%%%%%%%%%%%%%%%%%%%
%%%%%%%%%%%%%%%%%%%%%%%%%%%%%%%%%%%%%%%%%%%%%%%%%%%%%%%%%%%%%%%%

In this section we study the global hypoellipticity of the operator \eqref{Main-operator}, which is recalled by
\begin{equation*}
L  = \prod_{j=1}^{m}\left(D_t + c_j(t) P_j(D_x)\right), \ (t, x) \in \mathbb{T} \times \mathbb{T}^n.
\end{equation*}

Given that the main goal is to  exhibit connections  between the hypoellipticity of operators $L$ and 
\begin{equation*}
L_0 =  \prod_{j=1}^m \left( D_t + c_{0,j}P_j(D_x)\right),
\end{equation*} 
we start by observing that:

\begin{enumerate}
	\item[(i)] when each $L_j$ is globally hypoelliptic,  then $L$ and $L_0$ are globally hypoelliptic and the sets $Z_{{\mathcal{M}}_j}$ are finite;  
	
	\item[(ii)]  $L_0$ is  globally hypoelliptic  if and only if each $L_{0,j}$ is globally hypoelliptic;
	
	\item[(iii)]  let $\rho$ be a permutation of the set $\{1, \ldots, m\}$ and define 
	\begin{equation}\label{perm-operator}
	L_{\rho} = \prod_{j=1}^m L_{\rho(j)} = L_{\rho(1)} \circ \ldots \circ L_{\rho(m)}.
	\end{equation}
	If some $L_k$ is not globally hypoelliptic, then for any permutation $\beta$ of $\{1, \ldots, m\} \setminus \{k\}$ we obtain $L_{\beta} \circ L_k$ not (GH). This is the case for $L$ when $L_m$ is not globally hypoelliptic.
	
\end{enumerate}

\begin{prop}\label{GH-permutation}
	Consider $L_{\rho}$ an operator as in \eqref{perm-operator}.
	
	\begin{enumerate} 
		\item [(a)] If $L_{\rho}$ is globally hypoelliptic for any permutation $\rho$, then $L_{0}$ is also globally hypoelliptic. In particular, all sets $Z_{{\mathcal{M}}_k}$ are finite.
		
		\item [(b)] If some $L_{\rho}$ is  globally hypoelliptic, then at least one of the operators $L_{0,k}$ is globally hypoelliptic.  
	\end{enumerate}
	
\end{prop}

\begin{proof}
	Assume that every  $L_{\rho}$ is (GH) and consider the equation $L_k u = f \in C^{\infty}(\mathbb{T}^{n+1})$, for some $k$.  If  $\beta$ is a permutation of $\{1, \ldots, m\} \setminus \{k\}$, then the operator $L_{\beta} \circ L_{k}$ is (GH), which implies $u \in C^{\infty}(\mathbb{T}^{n+1})$, since $L_{\beta}(C^{\infty}(\mathbb{T}^{n+1})) \subset C^{\infty}(\mathbb{T}^{n+1})$. It follows from Theorem \ref{AGKM} that  $L_{0,k}$ is globally hypoelliptic and a) is proved.
	
	To verify b) it is enough to observe that if  $L_{\rho}$ is (GH), then $L_{\rho(m)}$ is necessarily globally hypoelliptic and consequently $L_{0,\rho(m)}$.
	
\end{proof}

We point out that, by Proposition \ref{GH-permutation},  when $L$ is globally hypoelliptic, then at least $L_{0,m}$ is  (GH). The next result improves this conclusion if  commutative conditions are admitted.

\begin{thm}\label{theorem-vari-1}
	Assume  that $L$ satisfies the commutative hypothesis
	\begin{equation*}
	[L_j, L_k] = L_j  L_k - L_k  L_j = 0, 
	\end{equation*}
	for every $j, k \in \{1, \ldots, m\}.$	 Then:
	
	\begin{enumerate}
		\item [(a)] $L$ is (GH) if and only if every $L_k$ is (GH);
		
		\item [(b)] if $L$ is (GH), then $L_0$ is (GH).
	\end{enumerate}
	
\end{thm}

\begin{proof}
	From the commutative hypothesis we have $L = L_{\rho}$, for any  permutation $\rho$ of $\{1, \ldots, m \}$. If $L$ is  globally hypoelliptic, then  each $L_k$ is also globally hypoelliptic by the preceding result. In particular, every $L_{0,k}$ is  (GH), and  consequently so is $L_0$.
	
\end{proof}

%%%%%%%%%%%%%%%%%%%%%%%%%%%%%%%%%%%%%%%%%%%%%%%%%%%%%%%%%%%%%%%%
%%%%%%%%%%%%%%%%%%%%%%%%%%%%%%%%%%%%%%%%%%%%%%%%%%%%%%%%%%%%%%%%
\subsection{H\"{o}rmander conditions \label{section3.1}}
%%%%%%%%%%%%%%%%%%%%%%%%%%%%%%%%%%%%%%%%%%%%%%%%%%%%%%%%%%%%%%%%
%%%%%%%%%%%%%%%%%%%%%%%%%%%%%%%%%%%%%%%%%%%%%%%%%%%%%%%%%%%%%%%%

In this section we study the global hypoelliticity problem by considering operators
$L_j$ satisfying  the following condition:  there exists positive  constants $\eta_j$ such that
\begin{equation}\label{HormanderCond}
\Im \mathcal{M}_j(t, \xi) \geq -\eta_j, \ t \in \mathbb{T}, \  \forall \xi \in \mathbb{Z}^n.
\end{equation}

Notice that there are no novelties if $\eta_j = 0$ is admitted, since in this case 
$\Im\mathcal{M}_j(\cdot, \xi)$ does not change sign for every $\xi$, and we may apply Theorem \ref{AGKM}. Otherwise, by assuming $\eta_j>0$, the existence of a globally hypoelliptic operator $L_j$,  with  $\Im\mathcal{M}_j(t,\cdot)$ changing sign, becomes possible, as follows from the next result.

\begin{thm}\label{Th-GH-H}
	An operator $L_j$ satisfying  \eqref{HormanderCond} is globally hypoellitpic if and only if $L_{0,j}$ is globally hypoelliptic. 
\end{thm}

Before the proof of this theorem, we present a direct consequence.

\begin{thm}\label{Theorem8}
	Admit that each operator $L_j$ satisfies  \eqref{HormanderCond}. If  $L_0$ is globally hypoelliptic, then  $L$ is globally hypoelliptic.
\end{thm}

\begin{proof}
	If $L_0$ is globally hypoelliptic, then each $L_{0,j}$ is also (GH) and, by the previous theorem, each $L_j$ is globally hypoelliptic.
	
\end{proof}

\begin{rem}
	We emphasize that the  phenomena of hypoellipticity with changes of sign 
	in the imaginary part $\Im \mathcal{M}_j(t, \xi)$,  for operators with symbols  satisfying the logarithm growth 
	$$
	|p_j(\xi)| = O(log(|\xi|)), \ \textrm{ when } \ |\xi|\to \infty,
	$$
	is discussed in \cite{AGKM}, as the reader can see in Sections 4 and 5 of that paper.
	
	In the next  example, we exhibit an operator satisfying \eqref{HormanderCond} which  cannot be captured by the approach presented in \cite{AGKM}.	
\end{rem}

\begin{ex}
	Let $P(D_x)$ be a pseudo-differential operator  on $\mathbb{T}$ defined by the symbol $p(\xi) =\alpha(\xi) + i\beta(\xi)$, where
	$$
	\alpha(\xi) = 
	\left\{
	\begin{array}{l}
	\xi^{-1}, \ \textrm{ if } \ \xi<0  \ \textrm{ is odd,} \\
	|\xi|, \ \textrm{ if } \ \xi<0  \ \textrm{ is even,} \\
	0, \ \textrm{ if  } \ \xi \geq 0,
	\end{array}
	\right. 
	\ \textrm{ and } \
	\beta(\xi) = 
	\left\{
	\begin{array}{l}
	1, \ \textrm{ if  } \ \xi<0, \\
	\xi, \ \textrm{ if } \ \xi \geq 0.
	\end{array}
	\right. 
	$$
	Clearly, both real and imaginary parts of $p(\xi)$ do not satisfy the logarithm condition.
	
	Now, let $a(t)$ and $b(t)$ be two real, smooth  and  positive functions on $\mathbb{T}$ with disjoint supports
	and define  $Q = D_t + (a(t)+ ib(t))P(D_x)$. In this case, we have $\Im \mathcal{M}(t, \xi) =0$ when $t \notin supp(a) \cup supp(b)$, for any $\xi \in \mathbb{Z}$. On the other hand, 
	$$
	t \in supp(a) \Rightarrow
	\Im \mathcal{M}(t, \xi) = a(t) \beta(\xi),
	$$
	and
	$$
	t \in supp(b) \Rightarrow
	\Im \mathcal{M}(t, \xi) = b(t) \alpha(\xi). 
	$$
	
	Note that $\Im \mathcal{M}(\cdot, \xi)$ changes sing if $\xi<0$ is odd. Indeed, 
	in this case $t \in supp(a)$ implies $\Im \mathcal{M}(t, \xi) = a(t)>0$,
	while $\Im \mathcal{M}(t, \xi) = b(t) \xi^{-1} <0$ when $t \in supp(b)$.

	Also, condition \eqref{HormanderCond} is fulfilled by choosing 
	$\eta = \max_{t \in \mathbb{T}}{b(t)}$. It follows that $Q$ is globally hypoelliptic if and only if $Q_0$ is globally hypoelliptic.
	
\end{ex}

To prove Theorem \ref{Th-GH-H} we make use of a standard result in literature:

\begin{lem}\label{lemma-smooth}
	Let $u \in \mathcal{D}'(\mathbb{T}^{n+1})$ be a distribution and  its $x$-Fourier coefficients be
	\begin{equation*}
	\widehat{u}(t, \xi) = (2 \pi)^{-n} <u(t, \cdot),  e^{-i x\xi}>,  \
	\xi \in \mathbb{Z}^n.
	\end{equation*}
	
	Given a sequence $\{c_{\xi}(t)\}_{\xi \in \mathbb{Z}^n}$ of smooth functions on $\mathbb{T}$, the formal series $u = \sum_{\xi \in \mathbb{Z}^{N}}{c_\xi(t) e^{i  x \cdot \xi}}$ 	converges in $\mathcal{D}'(\mathbb{T}^{n+1})$ if and only if for any $\alpha \in \mathbb{Z}_+$ there exist positive constants $N$, $C$ and $R$ such that
	\begin{equation}\label{part-smooth-coef}
	|\partial^{\alpha}_t c_{\xi}(t)| \leq C |\xi|^{-N}, \ |\xi|\geq R.
	\end{equation}
	Moreover, $u \in C^{\infty}(\mathbb{T}^{n+1})$ if and only if estimate \eqref{part-smooth-coef} holds true for every $N>0$. In both cases, 
	$c_{\xi}(\cdot) = \widehat{u}(\cdot, \xi)$.
\end{lem}

\subsection*{Proof of Theorem \ref{Th-GH-H}}

The necessary part  is a consequence of Theorem \ref{AGKM}. Conversely, let $u \in \mathcal{D}'(\mathbb{T}^{n+1})$ be a  solution of equation $L_ju = f \in C^{\infty}(\mathbb{T}{n+1})$. By taking the $x$-Fourier coefficients of $u$ and $f$ we get the O.D.E's	
\begin{equation}\label{edos}
\partial_t \widehat{u}(t, \xi) +    i{\mathcal{M}_j}(t, \xi) \widehat{u}(t,\xi) = \widehat{f}(t, \xi),  \ t \in \mathbb{T}, \ \xi \in \mathbb{Z}^n.
\end{equation}

Since $L_{0,j}$ is globally hypoelliptic, then  $\mathcal{M}_{0, j}(\xi)$ satisfies the  Siegel condition \eqref{Siegel} and  the set  $Z_{\mathcal{M}_j}$ is finite. When 	$\xi \notin Z_{\mathcal{M}_j}$, the solutions \eqref{edos} can be written in the form
\begin{equation}\label{Sol-4}
\widehat{u}(t, \xi) = \frac{1}{e^{ 2 \pi i{\mathcal{M}_{0, j}}(\xi)} - 1} \int_{0}^{2\pi}\exp\left(i\int_{t}^{t+s}\!\!{\mathcal{M}_j}(r, \, \xi) \, dr\right) \widehat{f}(t+s, \xi)ds.
\end{equation}

Recall that by  Theorem \ref{AGKM}, there are positive constants $C_1$, $N_1$ and $R_1$ such that
\begin{equation*}
|1-e^{2\pi i\mathcal{M}_{0,j}(\xi)}|^{-1}\leq C_1|\xi|^{N_1}, \ \textrm{for all} \  |\xi|\geq R_1.
\end{equation*}

Now, given $\alpha \in \mathbb{Z}_+$ and $\widetilde{N}>0$ there are positive constants $C_2$, $C_3$  and $R_2$ such that: $|p_j(\xi)|\leq C_2 |\xi|^{\nu_j}$ (see \eqref{bound-symb}) and
\begin{equation*}
\sum_{\beta=0}^{\alpha}\binom{\alpha}{\beta} |\partial_{t}^{\alpha -\beta} \widehat{f}(t, \xi)| \leq C_3 |\xi|^{-\widetilde{N}}, \ \forall |\xi|\geq R_2.
\end{equation*}

Hence, we get  
\begin{align*}
|\partial_t^{\alpha}  \widehat{u}(t, \xi)| & \leq C_1C_2C_3C_4|\xi|^{N_1 - N_2 + \nu_j}
\int_{0}^{2\pi}\exp\left(-\int_{t}^{t+s}\!\!\Im {\mathcal{M}_j}(r, \, \xi) \, dr\right) ds \\
& \leq 2 \pi C_1C_2C_3C_4|\xi|^{N_1 - \widetilde{N} + \nu_j + 2\pi \eta_j},
\end{align*}
with $C_4$ depending on ${\alpha}$ and $c_j(t)$.

Finally, for a fixed $\alpha \in \mathbb{Z}_+$ it is possible to obtain for any $N>0$, positive constants $C$ and $R$ such that
\begin{equation*}
|\partial_t^{\alpha}  \widehat{u}(t, \xi)| \leq C |\xi|^{-N}, \ t \in \mathbb{T}, \ |\xi|\geq R.
\end{equation*}

It follows from Lemma \ref{lemma-smooth} that $u \in C^{\infty}(\mathbb{T}^{n+1})$ and, finally, the proof is finished.
\qed

\begin{rem}
	
	We emphasize that condition \eqref{HormanderCond} may be replaced by 
	\begin{equation*}
	\Im \mathcal{M}_j(t, \xi) \leq \eta_j, \ t \in \mathbb{T}, \  \forall \xi \in \mathbb{Z}^n,
	\end{equation*}
	for $j \in \{1, \ldots m \}$, since 
	\begin{equation*}
	\widehat{u}(t, \xi) = \frac{1}{1 - e^{-  2 \pi i{\mathcal{M}_{0,j}}(\xi)}} \int_{0}^{2\pi}\exp\left(-i\int_{t-s}^{t}\!\!{\mathcal{M}_j}(r, \, \xi) \, dr\right) \widehat{f}(t-s, \xi)ds
	\end{equation*}
	is an equivalent expression for \eqref{Sol-4}.  In this case, 
	\begin{align*}
	\sup_{s \in[0, 2\pi]}\left|\exp\left(-i\int_{t-s}^{t}\!\!{\mathcal{M}_j}(r, \, \xi) \, dr\right) \right| & =  
	\sup_{s \in[0, 2\pi]}\left|\exp\left(\int_{t-s}^{t}\!\!{\Im \mathcal{M}_j}(r, \, \xi) \, dr\right) \right|  \\
	& \leq  e^{2\pi \eta_k}.
	\end{align*}	
\end{rem}

\subsection{A link with the solvability problem}

To introduce this section, let us consider the operator $L = L_1 \circ L_2$. 
It is already possible to state that $L$ is not globally hypoelliptic when $L_2$ is not globally hypoelliptic.  Hence, all that is needed is a search for conditions so that  the same conclusion holds true when  $L_1$ is not globally hypoelliptic.

This inquiry is  equivalent to the following problem: if $u$ is a singular solution of $L_1$, namely,  $u \in \mathcal{D}'(\mathbb{T}^{n+1}) \setminus C^{\infty}(\mathbb{T}^{n+1})$ and $L_1 u  \in C^{\infty}(\mathbb{T}^{n+1})$, under which conditions is there $u \in  L_2(\mathcal{D}'(\mathbb{T}^{n+1}))$?

Clearly, a solution for this question is connected with solvability  properties and we discuss this subject in the next few paragraphs. The first step consists in introducing the following:
\begin{defn}
	Let $\mathbb{E}_{j}$ be the space of distributions $v \in \mathcal{D}'(\mathbb{T}^{n+1})$ such that $\widehat{v}(\cdot, \xi) \in C^{\infty}(\mathbb{T})$, for each $\xi \in \mathbb{Z}^n$, and 
	$$
	\int_{0}^{2 \pi}\exp\left(i \int_{0}^{t}\mathcal{M}_j(r,\xi)dr \right) \widehat{v}(t, \xi)dt = 0, \ \textrm{if } \ \xi \in Z_{\mathcal{M}_j}.
	$$
\end{defn}

\begin{thm}\label{Th-Solv}
	Let $L_j$ be a globally hypoelliptic operator satisfying  \eqref{HormanderCond}. Thus, for each $v \in \mathbb{E}_{j}$, there exist $u \in \mathcal{D}'(\mathbb{T}^{n+1})$ such that $L_j u = v$.
\end{thm}

\begin{proof}
	By the hypothesis,  $L_{0,j}$ is globally hypoelliptic, $\mathcal{M}_{0, j}(\xi)$ satisfies  condition \eqref{Siegel} and the set  $Z_{\mathcal{M}_j}$ is finite. For $\xi \notin Z_{\mathcal{M}_j}$ it is possible to define
	$\widehat{u}(t, \xi)$ by expression \eqref{Sol-4},  while  in the case of $\xi \in Z_{\mathcal{M}_j}$ we set 
	\begin{equation}\label{Sol-6}
	\widehat{u}(t, \xi) = \exp\left( -i \int_{0}^{t} \mathcal{M}_j(r, \xi) dr \right) \int_{0}^{t}\exp\left(i \int_{0}^{s}\mathcal{M}_j(r,\xi)dr \right) \widehat{v}(s, \xi)ds.
	\end{equation}
	
	Since $v \in \mathbb{E}_{j}$, we obtain
	$\widehat{u}(\cdot, \xi) \in C^{\infty}(\mathbb{T})$ and
	$$
	\partial_t \widehat{u}(t, \xi) +    i{\mathcal{M}_j}(t, \xi) \widehat{u}(t,\xi) = \widehat{v}(t, \xi),  \ t \in \mathbb{T}, \ \xi \in \mathbb{Z}^n.
	$$
	
	Now, we shall prove that $\{\widehat{u}(t, \xi)\}_{\xi \in \mathbb{Z}^n}$ defines an element $u \in \mathcal{D}'(\mathbb{T}^{n+1})$. Firstly, since the set $Z_{\mathcal{M}_j}$ is finite, no estimates for \eqref{Sol-6} are needed. For the general case, an argument similar to the one in the proof of Theorem \ref{Th-GH-H}  shows that 
	\begin{align*}
	|\partial_t^{\alpha}  \widehat{u}(t, \xi)|  \leq 2 \pi C_1C_2C_3C_4|\xi|^{N_1 - N_2 + \nu_j + 2\pi \eta_j},
	\end{align*}
	where $C_3$ and $R_2$ now satisfy
	\begin{equation*}
	\sum_{\beta=0}^{\alpha}\binom{\alpha}{\beta} |\partial_{t}^{\alpha -\beta} \widehat{v}(t, \xi)| \leq C_3 |\xi|^{-N_2}, \ \forall |\xi|\geq R_2,
	\end{equation*}
	for some $N_2$. Then, it is possible to choose positive constants $C$, $N$ and $R$ such that
	\begin{equation*}
	|\partial_t^{\alpha}  \widehat{u}(t, \xi)| \leq C |\xi|^{-N}, \ t \in \mathbb{T}, \ |\xi|\geq R,
	\end{equation*}
	which implies  $u = \sum_{\xi \in \mathbb{Z}^{n}}{\widehat{u}(t, \xi) e^{i  x \cdot \xi}} \in \mathcal{D}'(\mathbb{T}^{n+1})$. Clearly, $L_ju = v$ and the proof is done.
	
\end{proof}

\begin{thm}\label{suf-non-hy}
	Admit that each $L_j$ satisfies  \eqref{HormanderCond}. If  some $L_k$ has a singular solution $u_k$ satisfying
	$$
	u_{k+j} = L_{k+j+1} u_{k+j+1}, \ j \in \{0, \ldots, m-k-1 \},
	$$
	then $L$ is not globally hypoelliptic. 
	
\end{thm}

\begin{proof}
	It is enough to consider the case $m=2$ and $k=1$. By hypothesis, we have $L_1 u_1 = f \in C^{\infty}(\mathbb{T}^{n+1})$ with $u_1 \notin C^{\infty}(\mathbb{T}^{n+1})$. 
	Since $u_1 \in \mathbb{E}_2$ we obtain $u_2 \in \mathcal{D}'(\mathbb{T}^{n+1})$ such that $L_2 u_2 = u_1$. Clearly $u_2 \notin C^{\infty}(\mathbb{T}^{n+1})$ and $f = L_1 \circ L_2 (u_2)$, hence the proposition is proved.
	
\end{proof}

\bibliography{references}

\end{document}